\documentclass[final,leqno]{siamltex}
\usepackage{chngcntr}

\usepackage[arrow, matrix, curve]{xy}
\usepackage[latin1]{inputenc} 	
\usepackage[T1]{fontenc}
\usepackage[english]{babel}
\usepackage{lmodern}
\usepackage{amsmath}
\usepackage{amssymb}  
\usepackage{mathtools}
\usepackage{mathrsfs}
\usepackage{euscript}
\usepackage{enumerate}
\usepackage{xspace}
\usepackage{hyperref}  
\usepackage{framed}
\usepackage{cite}
\usepackage{fancyhdr}
\usepackage{bbm}
\usepackage{float}
\usepackage{comment}
\usepackage{graphicx}
\usepackage{textcomp}
\usepackage{url}
\usepackage{srcltx} 
\usepackage{hyperref} 
\usepackage{microtype}
\usepackage{listings}

\DeclareFontShape{OMX}{cmex}{m}{n}{
  <-7.5> cmex7
  <7.5-8.5> cmex8
  <8.5-9.5> cmex9
  <9.5-> cmex10
}{}
\SetSymbolFont{largesymbols}{normal}{OMX}{cmex}{m}{n}
\SetSymbolFont{largesymbols}{bold}  {OMX}{cmex}{m}{n}

\hypersetup{draft=false, colorlinks=true, linkcolor=black, urlcolor=blue, citecolor=black}
\newtheorem{remark}[theorem]{Remark}

\newcommand{\R}{\mathbb{R}} 
\newcommand{\N}{\mathbb{N}} 

\newcommand{\K}{\mathcal{K}}
\renewcommand{\L}{\mathcal{L}}
\renewcommand{\K}{\mathcal{K}}
\newcommand{\M}{\mathcal{M}}
\newcommand{\A}{\mathcal{A}}
\newcommand{\J}{\mathcal{J}}

\newtheorem{defi}[theorem]{Definition}
\newtheorem{algo}[theorem]{Algorithm}
\newtheorem{ex}[theorem]{Example}

\author{Gesa Sarnighausen\thanks{Institute for Numerical and Applied Mathematics, University of G\"ottingen, Germany ({\tt g.sarnighausen@math.uni-goettingen.de}).} \and Anne Wald\thanks{Institute for Numerical and Applied Mathematics, University of G\"ottingen, Germany ({\tt a.wald@math.uni-goettingen.de})} \and Alexander Meaney \thanks{Department of Mathematics and Statistics, University of Helsinki, Finland ({\tt alexander.meaney@helsinki.fi}).} }

\title{Dynamic Computerized Tomography using Inexact Models and Motion Estimation}
        
\begin{document}

\maketitle

\textbf{Abstract.}
Reconstructing a dynamic object with affine motion in computerized tomography (CT) leads to motion artifacts if the motion is not taken into account. In most cases, the actual motion is neither known nor can be determined easily. As a consequence, the respective model that describes CT is incomplete. The iterative RESESOP-Kaczmarz method can - under certain conditions and by exploiting the modeling error - reconstruct dynamic objects at different time points even if the exact motion is unknown. However, the method is very time-consuming. To speed the reconstruction process up and obtain better results, we combine the following three steps:
\begin{enumerate}
\item RESESOP-Kacmarz with only a few iterations is implemented to reconstruct the object at different time points.
\item The motion is estimated via landmark detection, e.g. using deep learning.
\item The estimated motion is integrated into the reconstruction process, allowing the use of dynamic filtered backprojection. 
\end{enumerate}
We give a short review of all methods involved and present numerical results as a proof of principle.

\vspace*{2ex}


\vspace*{2ex}


\section{Introduction} 

Computerized tomography (CT) is an imaging technique that exploits the characteristic attenuation of X-radiation by different materials: By reconstructing the density inside an unknown object from measurements of the intensity loss, it is possible to visualize the space-dependent density of the object. CT thus requires the stable solution of an ill-posed inverse problem, see \cite{natterer}. 

The Radon operator $R$ - the forward operator of computerized tomography - explains the relation between the scanned object's density $f$ and the measured data $g$. This is usually formulated as an operator equation 
\begin{align*}
R f = g,
\end{align*}

\noindent
where the impact $g$ is called a sinogram, which is measured by the CT scanner. The inverse problem thus consists in computing the cause $f$ 
from the knowledge of $R$ and $g$.
The most popular method to solve this problem is filtered backprojection (FBP), which works well for static objects, see 
\cite{natterer, natterer2, PrinciplesCTI}. 

However, the scanned object may move or deform during the scanning process since it takes time to scan the object from all necessary angles. An overview about dynamic reconstruction methods can be found in \cite{bonnet03}.

Let the map $\Gamma: \Omega \to \Omega$ on the domain $\Omega$ of $f$ describe the motion. Then, instead of $R$, we consider the dynamic Radon operator $R^\Gamma := R \circ \Gamma$. If the motion is not taken into account, e.g.~in FBP, this typically leads to severe motion artifacts in the reconstruction. In most applications, the motion is not a quantity of interest, such that one usually aims at reconstructing a specific state of the object.

In practice, compensating for the motion can be important in medical imaging, e.g.~since the patient is breathing and the heart is beating. On the nano scale (nano-CT), manufacturing tolerances, small vibrations and object drift lead to a relative motion between the object and the scanner \cite{nanoCTLegierung}. 
If the motion is known, it is possible to incorporate it into the reconstruction, see, e.g., \cite{HahnDynamicFBP, Hahn2021, katsevich10} and the references therein. Otherwise, one either has to estimate the motion first, for example by tracking markers, and use the dynamic model for the inversion as in \cite{hahnMotionEstimation}, or by correcting the motion in the measured sinogram and use the standard model as in \cite{Lu2002}, or both the motion and the object have to be reconstructed jointly \cite{arridge_fernsel_hauptmann22, Chen19,Burger_2017}, which is particularly demanding if there is only a single projection per time step available. The problem gets even more involved if only sparse data is available \cite{Bubba17}.


In this article, we consider the standard CT setup in two dimensions with a parallel beam geometry. The angle of the tomograph and, by consequence, the position of the X-rays source, is given by $\varphi \in [0,2\pi]$ and the offset of an X-ray beam by $s \in [-1,1]$. The goal is to reconstruct a fixed state of the object from measurements affected by a rigid body motion. We combine two ways of including the motion in the model:
\begin{enumerate}
    \item Directly, by using the dynamic model
    \begin{displaymath}
        g^{\Gamma}(\varphi,s) = R^{\Gamma} f(\varphi,s) \coloneqq \int_{\R^2} f(\Gamma_{\varphi} x)\delta(s - x^T \theta(\varphi)) \,\mathrm{d}x,
    \end{displaymath}
    \noindent
    for $s \in \R, \varphi \in [0,2\pi]$, $\theta(\varphi) = ( \cos (\varphi),  \sin (\varphi))^\top$ and the Dirac's delta distribution $\delta$.
    \item Indirectly, by interpreting the static operator as an inexact version of the dynamic operator and adding the inexactness $\eta$ to the static model:
    \begin{displaymath}
          g^{\Gamma}(\varphi,s) = R f(\varphi,s) + \eta(\varphi,s) \coloneqq \int_{\R^2} f(x) \delta(s - x^T \theta(\varphi)) \,\mathrm{d}x + \eta(\varphi,s).
    \end{displaymath}
\end{enumerate}
\begin{figure}[ht]
\centering
\caption{Combined Method}
\includegraphics[width=0.7\textwidth]{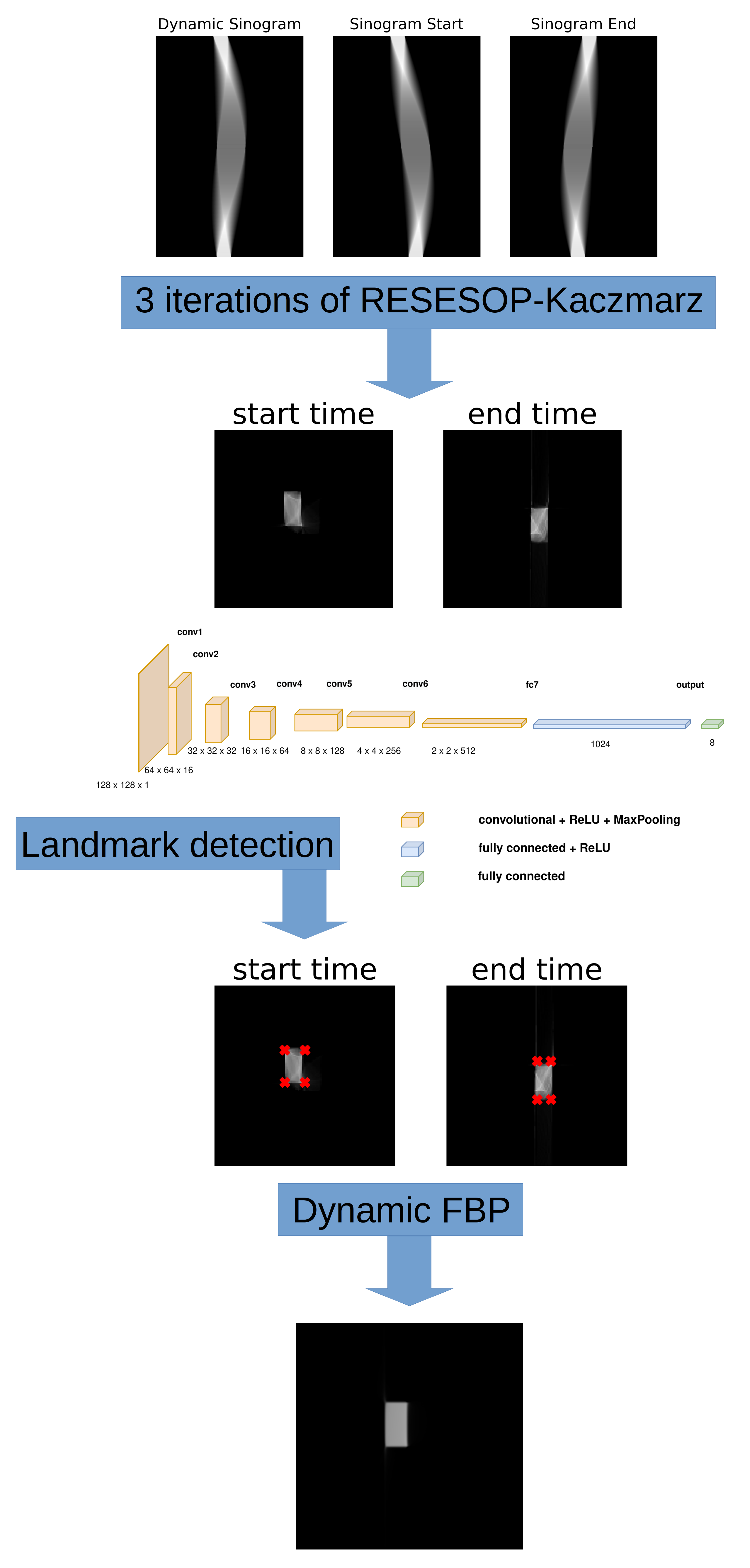}
\label{fig:pipeline}
\end{figure}

A class of adapted RESESOP-Kaczmarz algorithms \cite{RESESOP}, RESESOP being short for REgularizing SEquential Subspace OPtimization, can be used for the second approach. First, our dynamic CT problem is considered a semi-discrete inverse problem $R^{\Gamma}_{(\varphi,s)} f = g(\varphi,s)$ for all discrete tuples $(\varphi,s)$. The motion is interpreted as an inexactness in the forward model with respect to the static model, which needs to be estimated and is then compensated when computing the reconstruction. Hence, the subproblems are combined similarly as in Kaczmarz' iteration, but with an adapted RESESOP projection step instead of the usual projection onto the solution set of the respective subproblem. \\ 
Instead of knowing the motion, it is thus sufficient to estimate the inexactness of the model. The estimation of the inexactness is still an open task, but has certain benefits: the estimate needs not to be exact, but a small overestimation often leads to reasonable results, it can incorporate temporal and local information \cite{RESESOP}, and the respective reconstructions seem to yield a good basis for data-driven image processing \cite{luetjen23arxiv}. However, RESESOP-Kaczmarz is relatively slow due to the Kaczmarz loops which can not be parallelized here. 

To obtain a more efficient reconstruction technique under the assumption that an estimate of the inexactness is available, we propose a hybrid approach consisting of three steps for dynamic CT as illustrated in Figure \ref{fig:pipeline}:



\begin{enumerate}
\item \textbf{RESESOP-Kaczmarz}

First, the RESESOP-Kaczmarz algorithm is used with a very small amount of iterations and in a small resolution to calculate rather rough reconstructions of certain states of the object. 

\item \textbf{Landmark detection}

In each reconstruction, we identify landmarks of the object and track them to get an estimate of the actual motion. This task can be done by applying a trained deep neural network, but also manually.

\item \textbf{Dynamic Filtered Backprojection}

As the last step, the dynamic filtered backprojection algorithm is used with the estimated motion to get a better reconstruction of the object in high resolution.

\end{enumerate}

For a numerical proof of concept, we consider a simple directed rigid body motion of the motion, i.e., the object shifts in a fixed direction with a constant speed during data acquisition. 

The article is structured as follows. Section \ref{s:methods} explains the three methods we used. We start by presenting the RESESOP-Kaczmarz method in Section \ref{ss:Res-Kac}.
Section~\ref{ss:landmark} addresses the derivation of the underlying motion from landmark detection via deep learning or by hand.
and in Section \ref{ss:dynFBP}, the dynamic filtered backprojection algorithm for affine deformations is reviewed.
Section~\ref{s:results} contains the numerical results of the combined method introduced in this article for generated and real measured data.

\vspace*{2ex}

\section{Methods}
\label{s:methods}

\subsection{RESESOP-Kaczmarz}
\label{ss:Res-Kac}

The regularizing sequential subspace optimization (RESESOP) method is a regularizing iterative method to solve linear or nonlinear inverse problems in Hilbert and Banach spaces \cite{schoepferProjections, schoepferFast, aw18} and has been extended to problems with an inexact forward operator \cite{RESESOP}. It can especially be used for dynamic inverse problems if the static operator is considered the inexact version of the dynamic forward operator. The method combines the projection step from RESESOP with the semi-discrete setting from the Kaczmarz method and can be considered a kind of regularized gradient method with several search directions and a step size regulation by projecting onto suitable subsets of the source space.

Let ($X, \lVert \cdot \rVert_X$) 
and ($Y, \lVert \cdot \rVert_Y$) 
be real Hilbert spaces with their corresponding norms and inner products and 
$\mathcal{L}(X,Y)$ be the space of bounded linear operators from $X$ to $Y$. The inverse problem we want to solve is then 
\begin{align}
\label{eq:invProb}
    \A f = g, \quad \A \in \mathcal{L}( X, Y)
\end{align}
for measured data $g \in Y$ and the solution set $\mathcal{M}_{\A,g} \coloneqq \{f \in X : \A f = g\}$.

\vspace*{2ex}

\paragraph*{Sequential subspace optimization (SESOP)}
\label{s:sesop}
The SESOP method iteratively solves inverse problems by projecting onto suitable convex affine subspaces in a continuous setting with noise-free data and an exact forward operator.
To construct such subspaces, it is first necessary to define hyperplanes, half-spaces and stripes.

\vspace*{2ex}

\begin{defi}(Hyperplane, half-space, stripe)

\noindent
For $u \in X \setminus \{0\}$ and $ \alpha \in \R$ the set 
\begin{align*}
H(u,\alpha) \coloneqq \{f \in X :  \langle u , f \rangle_X = \alpha \}
\end{align*}
is called a \textit{hyperplane} in $X$ and the set
\begin{align*}
H_{\leq}(u,\alpha) \coloneqq \{f \in X :  \langle u , f \rangle_X \leq \alpha \}
\end{align*}
is called a \textit{half-space}. Analogously the half-spaces $H_{\geq}(u,\alpha)$, $H_{<}(u,\alpha)$ and $H_{>}(u,\alpha)$ are defined.

Now let additionally be $\xi \geq 0$. Then a \textit{stripe} can be defined as the set
\begin{align*}
H(u,\alpha, \xi) \coloneqq \{f \in X : \lvert \langle u , f \rangle_X - \alpha \rvert \leq \xi \}.
\end{align*}
Each stripe is bounded by the hyperplanes $H(u,\alpha + \xi)$ and $H(u,\alpha - \xi)$.
\end{defi}

\vspace{2ex}

With this notation, we are able to formulate 




\label{ss:sesop}

\begin{algo}{(SESOP)}

For $i \in I_n \subset \N_0$ a finite index set, the adjoint operator $\A^*$ of $\A$ and search directions $\A^*w_{n,i}$ for $w_{n,i} \in Y$, the $(n+1)$-th iterate $f_{n+1}$ can be calculated by 
\begin{align}
\label{SESOP}
f_{n+1} = P_{H_{n}}(f_{n})
\end{align}
as the metric projection $P_{H_{n}}$ of the $n$-th iterate $f_{n}$
onto the intersection of hyperplanes
\begin{align*}
H_{n} = \bigcap_{i \in I_n} H(\A^* w_{n,i}, \langle w_{n,i}, g \rangle_Y).
\end{align*}

\end{algo}

\vspace*{2ex}

\begin{remark}

For arbitrary elements $w \in Y$ the hyperplane 
$H(\A^*w, \langle w, g \rangle_Y) = \{ f \in X : \langle \A^*w, f \rangle_X = \langle w , g\rangle_Y \}$
includes $\mathcal{M}_{\A,g}$, since for every $z \in \mathcal{M}_{\A,g}$ holds
\begin{align}
\label{eq:sesopOnHyperplane}
\langle \A^*w, z \rangle_X = \langle w, \A z \rangle_Y = \langle w , g \rangle_Y.
\end{align}
This gives an intuition as to why the hyperplanes are chosen as above.
\end{remark}

\vspace*{2ex}

\noindent
The calculation of the next iterate $f_{n+1}$ can also be considered an update to the current iterate $f_n$ with multiple search directions per iteration and regulation of the step size:

\begin{lemma}{(Calculation of projection onto intersection of hyperplanes)}
\label{lem:projectionIntersection}

\noindent
Let $H = \bigcap_{i = 1}^{n} H(u_i,\alpha_i)$ be the intersection of hyperplanes for $u_i \in X$ and $\alpha_i \in \R$ for $i = \{1,\dots,n\}, n\in \N$. Then the projection of $f \in X$ onto $H$ can be determined as

\begin{align}
\label{eq:projection}
P_H(f) = f - \sum_{i = 1}^n \tilde{t_i} u_i,
\end{align}

where $\tilde{t} = (\tilde{t_1},\dots,\tilde{t_n})$ minimizes the function

\begin{align*}
h(t) = \frac{1}{2} \big\lVert f - \sum_{i = 1}^n t_i u_i \big\rVert_X^2 + \sum_{i = 1}^n t_i \alpha_i.
\end{align*}
If the elements $u_1,\dots ,u_n$ are linearly independent, $h$ has a unique minimal solution $\tilde{t}$.

\end{lemma}

\vspace*{2ex}

\noindent
A proof of Lemma \ref{lem:projectionIntersection} can be found in Proposition 3.12 in \cite{schoepferProjections}.
If the search directions $\A^*w_{n,i}$ for $i \in I_n$ and the finite index set $I_n$ are chosen correctly, the iterates $f_n$ converge to a solution $f$ with $\A f = g$, see e.g. Section 6.2.2 of \cite{RegMethodsBanachSpace}.

\vspace*{2ex}

\paragraph*{RESESOP}
\label{ss:resesop}
RESESOP is the regularizing version of SESOP and can take into account noisy data and an inexact forward operator. 

Let the global modelling error $\eta > 0$ of the inexact forward operator $\A^\eta$ and the global noise level $\delta$ of the noisy data $g^\delta$ be given by
\begin{align*}
    \lVert \A^\eta f - \A f \rVert \leq \eta \rho \text{ for all } f \in B_\rho (0) \subset X, \quad \lVert g - g^\delta \rVert_Y \leq \delta
\end{align*}
\noindent
for $\rho > 0$. 

To regularize the SESOP method, the projection onto hyperplanes is then changed to a projection onto stripes
\begin{align*}
       H_{n,i}^{\eta,\delta} := \big\lbrace f \in X \ : \ \left\lvert\left\langle (\A^{\eta})^* w_{n,i}^{\eta,\delta}, f \right\rangle - \left\langle w_{n,i}^{\eta,\delta},g^{\delta} \right\rangle\right\rvert \leq (\delta+ \eta\rho) \lVert w_{n,i}^{\eta,\delta} \rVert  \big\rbrace
\end{align*}

\noindent
As in the case of exact data, the stripes contain the solution set $\M_{\A f=g}$.

The discrepancy principle \cite{inverseCrime} assumes that in the inexact case the solution cannot get better than the noise level and modeling error and yields a finite stopping index $n^*$ as soon as 
\begin{align*}
\lVert \A^{\eta} f_{n^*} - g^{\delta} \rVert_Y \leq \tau (\eta \rho + \delta)
\end{align*}
for a fixed $\tau > 1$.
A proof of the convergence of this version of RESESOP can be found in \cite{RESESOP}. For the case of error-free forward operators, see, e.g., \cite{RegMethodsBanachSpace}.

\vspace*{2ex}

\paragraph*{RESESOP-Kaczmarz}
\label{s:resesop-kacz}
We summarize some elementary results from \cite{RESESOP}. 
Let $\A_{j,m} \in \mathcal{L}(X,Y_{j,m})$ and $\A_{j,m}^\eta \in \mathcal{L}(X,Y_{j,m})$ be the inexact version of $\A_{j,m}$.

In a semi-discrete setting, we can use the RESESOP-Kaczmarz method to solve a system of equations
\begin{align*}
    \A_{j,m} f &= g_{j,m}, \\ 
    \lVert \A_{j,m}^\eta - \A_{j,m} \rVert_{\mathcal{L}(X,Y_{j,m})} &\leq \eta_{j,m}, \\ 
    \lVert g^\delta_{j,m} - g_{j,m}, \rVert_{Y_{j,m}} &\leq \delta_{j,m}
\end{align*}

\noindent
for $j \in \J$, $m \in \M$ and finite index sets $\J$ and $\M$.

Combining RESESOP with Kaczmarz's method allows the inclusion of local modeling errors and local noise in the reconstruction process.
This is especially an advantage for dynamic computer tomography, since the model error of the dynamic forward operator changes with time.

\vspace*{2ex}

\begin{remark}
\label{ex:diffVersion}
In case of computerized tomography with $K \in \mathbb{N}$ source positions $\varphi_k \in [0,2\pi]$, $k \in \K \coloneqq \{0,\dots K-1\}$, and $L \in \mathbb{N}$ detector points $s_l \in [-1,1]$, $l \in \mathcal{L} \coloneqq \{0,\dots L-1\}$, the algorithm allows four different modalities:
\begin{enumerate}[V1:]
\item $\J = \K \times \mathcal{L}$ and $\M = \{0\}$ results in Kaczmarz projections with respect to both source positions and detector points. 
\item $\J = \K $ and $\M = \mathcal{L}$ results in Kaczmarz projections with respect to the source positions and averaging with respect to the detector points.
\item $\J = \mathcal{L}$ and $\M = \K$ results in Kaczmarz projections with respect to the detector points and averaging with respect to the source directions.
\item $\J = \{0\}$ and $\M =  \K \times \mathcal{L}$ results in averaging with respect to both detector points and search directions.
\end{enumerate}
In this contribution, we use the version V1, which is the most precise algorothm in the sense that each subproblem is considered individually. We also emphasize that the modalities V2, V3 and V4 allow for (partial) parallelization of the code in contrast to V1.
\end{remark}

\vspace*{2ex}

\begin{algo}{(RESESOP-Kaczmarz)}
\label{algo:RESESOP_Kaczmarz}

\noindent
Let $[n] = n \text{ mod } J$. Fix an initial value $f_0^{\eta,\delta} \in \mathcal{B}_\rho(0)$ and precision constants $\tau_{j,m} > 1$ for $j \in \J$ and $m \in \M$.

While the discrepancy principle is not yet fulfilled for all $m \in \M$ and the current iterate is not close enough to the solution of the current subproblem, i.e., $D_{n}^{\eta,\delta} \neq \emptyset$ with
\begin{align*}
D_{n}^{\eta,\delta} \coloneqq \{ m \in \M : \lVert \A^\eta_{[n],m} f_n^{\eta,\delta} - g^\delta_{[n],m} \rVert_{Y_{[n],m}} > \tau_{[n],m} (\eta_{[n],m} \rho + \delta_{[n],m}) \}
\end{align*}
for each $n \in \N_0$,
perform the next iteration step

\begin{align*}
f_{n+1}^{\eta,\delta} \coloneqq P_{H_{n}^{\eta,\delta,\rho}}\left(f_n^{\eta,\delta}\right), \quad H_{n}^{\eta,\delta,\rho} \coloneqq \bigcap_{i \in I_n} H\left( u_{n,i}^{\eta,\delta}, \alpha_{n,i}^{\eta,\delta}, \xi_{n,i}^{\eta,\delta} \right).
\end{align*}

\noindent
We choose for $n \in \N$ and $i \in I_n^{\eta,\delta} \subset \{n-N+1,\dots,n\} \cap \N_0$, $N \in \N \setminus \{ 0 \}$
\begin{align*}
u_{n,i}^{\eta,\delta} &\coloneqq \sum_{m \in \M} \left(A_{[i],m}^\eta\right)^* w_{n,i}^{\eta,\delta,m} \\
 \alpha_{n,i}^{\eta,\delta} &\coloneqq  \sum_{m \in \M} \langle w_{n,i}^{\eta,\delta,m} , g^\delta_{[i],m} \rangle_Y \\
  \xi_{n,i}^{\eta,\delta} &\coloneqq \sum_{m \in \M} (\delta + \eta \rho) \lVert w_{n,i}^{\eta,\delta,m} \rVert_{Y_{[i],m}}\\
 w_{n,i}^{\eta,\delta,m} &\coloneqq 
\begin{cases}
\A^\eta_{[i],m} f_n^{\eta,\delta} - g^\delta_{[i],m} & \text{if }  m \in D_n^{\eta,\delta}\\
0 & \text{if } m \notin D_n^{\eta,\delta}
\end{cases}.
\end{align*}
\noindent
One full iteration over all $J$ subproblems is completed if $[n] = 0$.
Stop the algorithm after a full iteration at the index $[n^*] = 0$ when $f_{n^*}^{\eta,\delta} = f_{n^*-J}^{\eta,\delta}$. 
This means that the set $D^{\eta,\delta}_n$ has been empty at least for $n \in \{n^*,\dots,n^*-J\}$ and therefore the reconstruction cannot get better. 
\end{algo}

\vspace*{2ex}

\begin{theorem}{(Convergence of SESOP-Kaczmarz)}

\noindent
Let $\{f_n\}_{n\in\N}$ be the sequence of iterates generated by the RESESOP-Kaczmarz algorithm with an initial value $f_0 \in X$ and exact data, i.e., $\eta = \delta = 0$. 

If the optimization parameters $\lvert t_{n,i} \rvert \leq t$ for $i \in I_n$ are bounded by a constant $t > 0$, then $\{f_n\}_{n\in\N}$ converges strongly to a solution $f$ of the semi-discretized problem $\A_{j,m} f = g_{j,m}$ for $j \in \J$, $m \in \M$.

\end{theorem}

\vspace*{2ex}

\begin{theorem}{(Convergence and Regularization of RESESOP-Kaczmarz)}

\noindent
Let all assumptions from the previous theorem be true and the set of search directions be linearly independent. Then
\begin{enumerate}
\item 
the RESESOP-Kaczmarz algorithm with the adapted discrepancy principle as a stopping criterion leads to a finite stopping index $n_*$,
\item
for each iteration $n \in \N$ holds $\lim_{(\eta,\delta) \to (0,0)} f_n^{\eta,\delta} = f_n$,
\item

$\lim_{\eta,\delta \to (0,0)} f^{\eta,\delta}_{r_*} =  P_{M_{\A_{j,m} f = g_{j,m}} \cap \mathcal{B}_{ \rho}(0)}(f_0)$.
\end{enumerate}

\end{theorem}

\vspace*{2ex}

Proofs for both theorems can be found in \cite{RESESOP}.

\vspace*{2ex}

We give a pseudo-code for implementing RESESOP-Kaczmarz, discretized with respect to both source positions and detector points:

\vspace*{2ex}

\begin{lstlisting}[numbers=left, caption={Pseudo code for V1 of RESESOP-Kaczmarz with two search directions }, basicstyle = \footnotesize, label=alg1:resesop, mathescape=true, language=Pascal,lineskip=2pt]
$\mathbf{Input}$: $\K \text{ set of source positions, } \mathcal{L} \text{ set of detector points, }\tau_{k,l} > 1 \text{ for } k \in K, l \in L,$
$ g^{\delta} \text{ perturbed sinogram, } \eta,\delta \text{ model/data error for each } k \in \K, l \in \mathcal{L}, \text{ max\_iterations},$

$\text{stopped}_{k,l}$ = $1 \text{ for all } k \in \K, l \in  \mathcal{L}$
$f, u_{\text{old}}$ = $0$
$\alpha_{\text{old}},\xi_{\text{old}}, \text{full\_iter}$ = $0$
while $\exists k,l: \text{stopped}_{k,l} \neq 0 \text{ and } \text{full\_iter} < \text{max\_iterations}$ do
    for $k \in K$ do
        for $l \in L$ do
            $\text{res}$ = $\A^{\eta}_{k,l} f - g^{\delta}_{k,l}$
            if $\lvert \text{res} \rvert$ $\leq$ $\tau_{k,l} \cdot (\eta_{k,l} + \delta_{k,l})$ then
                $\text{stopped}_{k,l}$ = $0$
            else
                $u_{\text{new}}$ = $(\A^{\eta}_{k,l})^* \cdot \text{res}$
                if $\lVert u_{\text{new}} \rVert^2$ > $0$ then
                    $\alpha_{\text{new}}$ = $\text{res} \cdot g^{\delta}_{k,l}$
                    $\xi_{\text{new}}$ = $\lvert \text{res} \rvert \cdot (\eta_{k,l} + \delta_{k,l}) $
    	            # determine $\tilde{f} = P_{H(u_{\text{new}},\alpha_{\text{new}}+\xi_{\text{new}})}(f)$    	    
                    $\tilde{f}$ = $f - \frac{\lvert \text{res} \rvert (\lvert \text{res} \rvert - (\eta_{k,l} + \delta_{k,l}))}{\lVert u_{\text{new}} \rVert^2}  u_{\text{new}}$ 
                    # determine $f = P_{H(u_{\text{new}},\alpha_{\text{new}}+\xi_{\text{new}}) \cap H_{\gtrless}(u_{\text{old}},\alpha_{\text{old}}\pm \xi_{\text{old}})}(\tilde{f})$        
                    if $\tilde{f} \in H(u_{\text{old}},\alpha_{\text{old}}+\xi_{\text{old}})$ then
                        $f$ = $\tilde{f}$
                    else if $\lVert u_{\text{new}} \rVert^2 \cdot \lVert u_{\text{old}} \rVert^2 - \langle u_{\text{new}}, u_{\text{old}} \rangle^2 \neq 0$
                        $\text{Decide whether}$
                        1. $\tilde{f} \in H_{>}(u_{\text{old}},\alpha_{\text{old}}+ \xi_{\text{old}}) \text{ or}$
                        2. $\tilde{f} \in H_{<}(u_{\text{old}},\alpha_{\text{old}} - \xi_{\text{old}})$                                
                        $t$ = $\frac{\langle u_{\text{old}}, \tilde{f} \rangle - (\alpha_{\text{old}} \pm \xi_{\text{old}})}{\lVert u_{\text{new}} \rVert^2 \cdot \lVert u_{\text{old}} \rVert^2 - \langle u_{\text{new}}, u_{\text{old}} \rangle^2}$                
                        $f$ = $\tilde{f} + \langle u_{\text{new}}, u_{\text{old}} \rangle \cdot t \cdot u_{\text{new}} - \lVert u_{\text{new}} \rVert^2 \cdot t \cdot u_{\text{old}} $
                    else:
                       $f$ = $\tilde{f}$
                    $\text{set negative values in } f  \text{ to zero}$
                    $\alpha_{\text{old}}$ = $\alpha_{\text{new}}$
                    $\xi_{\text{old}}$ = $\xi_{\text{new}}$
                    $u_{\text{old}}$ = $u_{\text{new}}$
        end
    end
    $\text{full\_iter}$ = $\text{full\_iter} + 1$
end
$\mathbf{Output}$: $f$
\end{lstlisting}

\vspace*{2ex}

\subsection{Motion Estimation via Landmark Detection}
\label{ss:landmark}
For the purpose of a proof of concept, we consider homogeneous rectangular objects and restrict ourselves to the rather simple case of an affine linear motion $\Gamma_\theta$ 
with constant speed of the form
\begin{align}
\label{eq:affineMotion}
\begin{split}
\Gamma_\theta x &= A_\theta x + b_\theta, \quad \theta \in S_1, \\
A_\theta &= A(t_\theta) = I + \frac{t_\theta}{N-1} (A - I), \\
b_\theta &= b(t_\theta) =  \frac{t_\theta}{N-1} b
\end{split}
\end{align}
for the unit circle $S_1$, the identity matrix $I \in \R^{2\times 2}$, $A \in \R^{2 \times 2}$, $b \in \R^2$, $N \in \N$ the number of time points (i.e., the number of scanning angles $\varphi$), and $t_\theta \in \{0, \ldots, N-1\}$.

In this case, it is sufficient to know the coordinates of four landmarks of the object, each at the start and end time.
By plugging these into~\eqref{eq:affineMotion} we get a linear system of eight equations containing six unknowns for determining the motion. The motion estimation then consists in computing the least-squares solution of this overdetermined system.

\begin{figure}[t]
\centering
\caption{CNN6 architecture}
\includegraphics[width=\textwidth]{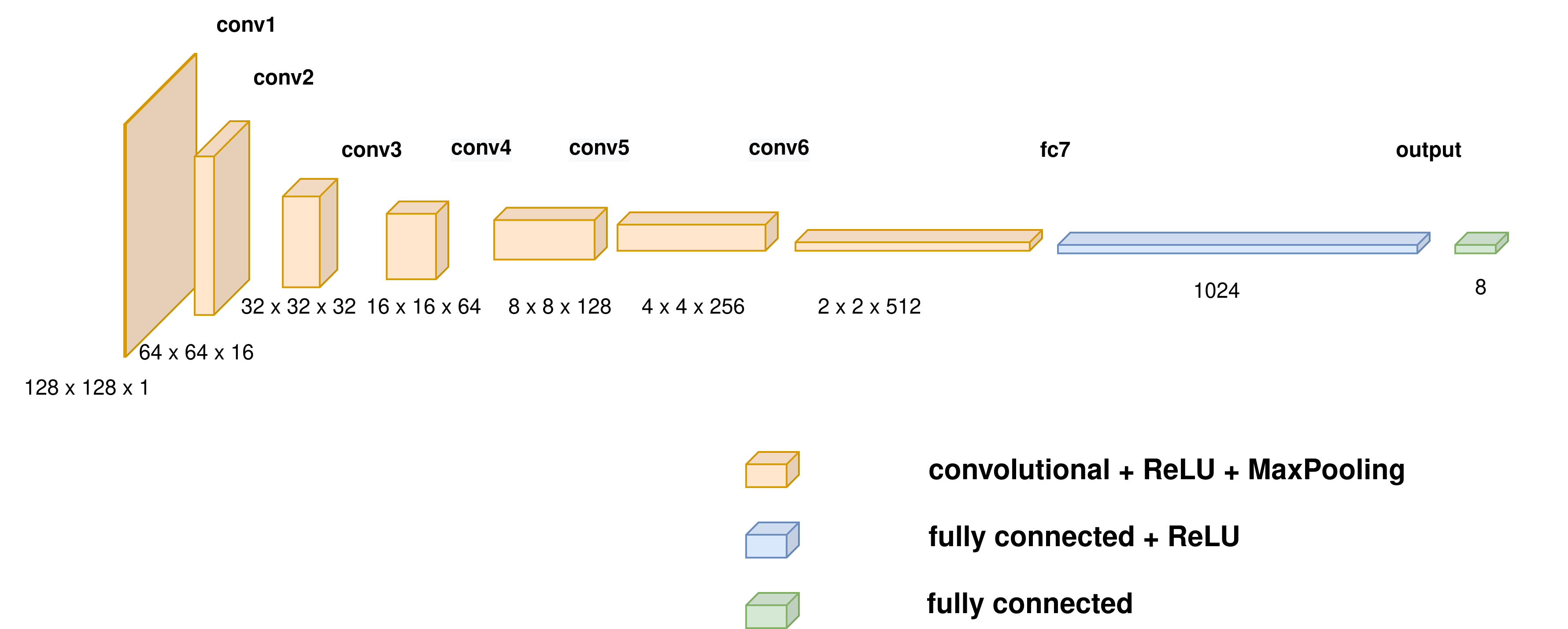}
\label{fig04:cnn6}
\end{figure}

\vspace*{2ex}

\subsubsection{Data-driven landmark detection}
\label{sec:cnn6}
As a proof-of-concept, we trained a small network for the purpose of landmark detection for randomly generated rectangles, i.e., it is trained to determine the coordinates of the four corners of a rectangle.

\vspace*{2ex}

\paragraph*{Architecture}
The implemented architecture, denoted by CNN6 (see Figure~\ref{fig04:cnn6}), was inspired by \cite{cnn6} where it was originally used for facial landmark detection.


\vspace*{2ex}

\paragraph*{Training Data Set}
\label{ss:dataset}
The network was trained in two steps with two different data sets, each containing 500 pairs of images of size $128 \times 128$ as an input for the network and the corresponding coordinates of the landmarks as the output: 
For pre-training, a data set with uncorrupted rectangular phantoms was used. 

After pre-training, the net was trained with a data set containing the reconstructions of the inital state of moving rectangles by the RESESOP-Kaczmarz method with three iterations. 

\vspace*{2ex}

\paragraph*{Hyperparameters}

As a regularization technique, we used early stopping with a patience of 200 epochs and allowed a maximum of 10~000 epochs. \\
We remark that the network in the original article~\cite{cnn6} was trained for 120~000 epochs, leading to our choice of this relatively high patience. 

We used Adam \cite{adam} as an optimizer with an initial step size of $10^{-6}$, a batch size of 8, and the mean squared error as a loss function. 


\vspace*{2ex}

\paragraph*{Training Results}
\label{ss:trainRes}

The pre-training took 10 000 epochs and already reached a good precision as can be seen in figure~\ref{fig04:trainLoss}.

\begin{figure}[t]
\centering
\caption{Pre-training (left) and training loss (right)}
\includegraphics[width=0.46\textwidth]{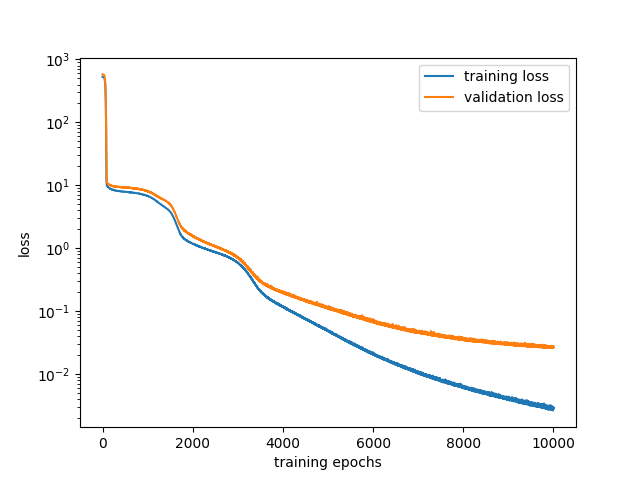}
\includegraphics[width=0.49\textwidth]{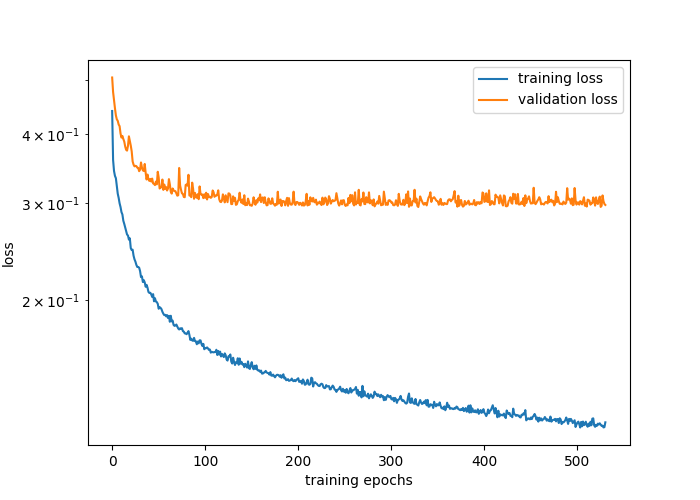}
\label{fig04:trainLoss}
\end{figure}

The second phase of the training with the reconstructed data set only took 500 epochs until the early stopping ended the training. 
Since the reconstructions are not as sharp and precise as the constructed phantoms, it is not surprising that the validation loss does not get as small as on the pre-training data set.

\noindent


\noindent





\noindent



\vspace*{2ex}

\subsection{Dynamic Filtered Backprojection}
\label{ss:dynFBP}

To derive a dynamic filtered backprojection algorithm for affine deformations, the method of the approximate inverse can be used.
A detailed introduction to the method of the approximate inverse can be found in \cite{ApproximateInverse}.
This method can be used to solve operator equations
\begin{align}
\A f = g
\end{align}

\noindent
for a linear invertible operator $\A : X \to Y$ with $X = L^2(\Omega_1, \mu_1)$ and $Y = L^2(\Omega_2, \mu_2)$ for $\Omega_i \subset \R^{n_i}$ open, bounded domains, $\mu_i$ measures on $\Omega_i$, $i\in \{1,2\}$ and $L^2$ the space of Lebesgue square-integrable functions. 

To reconstruct $f$, a smoothed version $f^\gamma$ with

\begin{align}
\label{eq:f_gamma}
f(x) \approx f^\gamma (x) = \langle f, \delta^\gamma(\cdot, x) \rangle_X
\end{align}

\noindent
is calculated for a suitable mollifier $\delta^\gamma(y,x) \in L^2(\Omega_1 \times \Omega_1, \mu_1 \times \mu_1)$.

\vspace*{2ex}

\begin{defi}{(Mollifier)}

\noindent
A function $\delta^\gamma (\cdot, x) \in L^2(\R^{n_1},\mu_1)$ for $x \in \R^{n_1}$ is called a \textit{mollifier} if
\begin{itemize}
\item for all $\gamma > 0$ it holds
\begin{align*}
\int_{\R^{n_1}} \delta^\gamma (y,x) d\mu_1(y) = 1.
\end{align*}

\item for $\gamma \to 0$ $f_\gamma$ converges to $f$ in $L^2(\Omega_1, \mu_1)$.
\end{itemize}
\end{defi}

\vspace*{2ex}

\noindent
With equation~\eqref{eq:f_gamma} it is not yet possible to calculate $f_\gamma$ since the exact solution $f$ appears in the expression. This problem can be overcome: 

For a mollifier $\delta^\gamma(\cdot,x)$ and a solution $\Psi^\gamma(x) \in Y$ of the auxiliary problem
\begin{align}
\label{eq:auxiliaryProblem}
\A^* \Psi^\gamma(x) = \delta^\gamma(\cdot,x)
\end{align}

\noindent 
we find that
\begin{align*}
f^\gamma(x) 
= \langle f, \delta^\gamma(\cdot, x) \rangle_X
= \langle f, \A^* \Psi^\gamma(x) \rangle_X
= \langle \A f, \Psi^\gamma(x) \rangle_Y
= \langle g, \Psi^\gamma(x) \rangle_Y.
\end{align*}

\vspace*{2ex}

\begin{remark}
Since the operator $\A$ is invertible, $\delta^\gamma(\cdot,x)$ is in the range of the adjoint operator $\A^*$ and therefore the auxiliary problem~\eqref{eq:auxiliaryProblem} has a solution $\Psi^\gamma(x) \in Y$.
\end{remark}

\vspace*{2ex}

Now the approximate inverse depending on the chosen mollifier can be defined as follows:

\vspace*{2ex}

\begin{defi}{(Approximate Inverse)}

\noindent
Let $\A: X \to Y$ be a linear, bounded and invertible operator and $\delta^\gamma$ be a mollifier.
Then the linear mapping $\tilde{\A}_\gamma : Y \to X$ with
\begin{align*}
\tilde{\A}_\gamma g(x) = \langle g, \Psi^\gamma(x) \rangle_Y
\end{align*}

\noindent
is called the \textit{approximate inverse} of $\A$.
The solution $\Psi^\gamma(x)$ of~\eqref{eq:auxiliaryProblem} is called the \textit{reconstruction kernel} associated to $\delta^\gamma$.

\end{defi}

\vspace*{2ex}

Under certain conditions on $ \Psi^\gamma(x)$ the approximate inverse 
$\tilde{\A}_\gamma$ is continuous and 
$\tilde{\A}_\gamma g$ converges pointwise to 
$\A^{-1}_\gamma g$ for $\gamma \to 0$. Further details can be found in the second chapter of \cite{ApproximateInverse}.

An advantage of the computation of reconstruction kernels is that the auxiliary problem~\eqref{eq:auxiliaryProblem} does not depend on the noisy measured data $g^\delta$. Therefore the reconstruction kernel is independent of noise.

On the other hand, the computation of the reconstruction kernel is time-consuming since it needs to be calculated for each reconstruction point $x$ in the domain of $f$.
However, it is possible to simplify the calculation under certain invariance conditions of operator $\A$. This is shown in the next lemma (see Lemma 3.2 in \cite{HahnDynamicFBP} and Theorem 2.6 in \cite{ApproximateInverse}).

\vspace*{2ex}

\begin{lemma}
\label{lem:simplify}
Let $T_1^x: X \to X$ and $T_2^x: Y \to Y$ for $x \in \Omega_1$ be linear operators with
\begin{align}
\label{eq:condSimplify}
T_1^x \A^* = \A^* T_2^x
\end{align}

\noindent
and let the mollifier $\delta^\gamma$ be generated by $T_1^x$, meaning that there exists $x^* \in \Omega_1$ with
\begin{displaymath}
\delta^\gamma(\cdot, x) = T_1^x \delta^\gamma(\cdot,x^*).
\end{displaymath}

\noindent
If $ \Psi^\gamma(x^*)$ solves $\A^* \Psi^\gamma(x^*) = \delta^\gamma(\cdot, x^*)$, then 
$\Psi^\gamma(x) = T_2^x  \Psi^\gamma(x^*)$
solves 
\begin{align*}
\A^* \Psi^\gamma(x) = \delta^\gamma(\cdot, x)
\end{align*}

\noindent
for any $x \in \Omega_1$.

\end{lemma}

\vspace*{2ex}

\begin{proof}
It holds
\begin{align*}
\delta^\gamma(\cdot, x) 
= T_1^x \delta^\gamma(\cdot,x^*)
= T_1^x \A^* \Psi^\gamma(x^*)
= \A^* T_2^x \Psi^\gamma(x^*).
\end{align*}

\noindent
Therefore, $T_2^x \Psi^\gamma(x^*)$ solves the auxiliary problem $\A^* \Psi^\gamma(x) = \delta^\gamma(\cdot, x)$.
\end{proof}

\vspace*{2ex}

If Lemma~\eqref{lem:simplify} can be applied, only one auxiliary problem has to be solved. In this case, the method of the approximate inverse can be formulated as 
\begin{align}
\label{eq:ApproxInvSimplified}
\tilde{\A}_\gamma g(x) = \langle g, T_2^x \Psi_\gamma(x^*) \rangle_Y
\end{align}

\noindent
Now we can apply the method of the approximate inverse to dynamic computerized tomography as it is done in \cite{HahnDynamicFBP}:

\vspace*{2ex}

\begin{lemma}[Application to Affine Deformations in Dynamic CT]

Let the motion be given by
\begin{align*}
\Gamma_\theta x = C_\theta x + b_\theta, \quad \theta \in S_1,
\end{align*}

\noindent
for $b_\theta \in \R^2$ and $C_\theta \in \R^{2 \times 2}$ with continuously differentiable entries, $\det C_\theta \neq 0$ for all measured angles $\theta =  (\cos(\varphi), \sin(\varphi))^T$ and
\begin{align*}
h(\theta) := (C_\theta^{-*} \theta)_1 \cdot \frac{\partial}{\partial \varphi} (C_\theta^{-*} \theta)_2 - (C_\theta^{-*} \theta)_2 \cdot \frac{\partial}{\partial \varphi} (C_\theta^{-*} \theta)_1 \neq 0,
\end{align*}
where $C_\theta^{-*}:= \big(C_\theta^{-1}\big)^*$.

Then the auxiliary problem can be solved analytically by using the inverse dynamic Radon operator
\begin{align*}
(R^\Gamma)^* \Psi_x^\gamma = \delta_x^\gamma
= (R^\Gamma)^{-1} (R^\Gamma) \delta_x^\gamma
=  (R^\Gamma)^* \left( \frac{1}{4 \pi}
\lvert \text{det} C_\theta \rvert^2 \lvert h(\theta) \rvert \ell^{-1}(R^\Gamma) \delta_x^\gamma \right)
\end{align*}

yielding the reconstruction kernel
\begin{align}
\label{eq:recKernel}
\Psi_x^\gamma(\theta,s) = \frac{1}{4 \pi}
\lvert \text{det} C_\theta \rvert^2 \lvert h(\theta) \rvert \ell^{-1}(R^\Gamma) \delta_x^\gamma
\end{align}

with the Riesz potential $\ell$ defined by $\mathcal{F} \ell^{-1} g(\sigma) = \lvert \sigma \rvert \mathcal{F} g(\sigma)$ and the two-dimensional Fourier transform $\mathcal{F}$ given by $\mathcal{F} f (\xi) = \frac{1}{2 \pi} \int_{\R^2} f(x) e^{-ix^T\xi} \, \mathrm{d}x$.
\end{lemma}

\vspace*{2ex}

A derivation of this result can be found in Section 4 of \cite{HahnDynamicFBP}.

\vspace*{2ex}

\noindent
Since the dynamic Radon operator fulfills the invariance condition
\begin{displaymath}
T^x (R^\Gamma)^* = (R^\Gamma)^* T^{(0,(C_\theta^{-1} x)^T \theta)}
\end{displaymath}
from Lemma~\ref{lem:simplify} for the shift operators
\begin{align*}
T^x f(y) &= f(y-x) \\
T^{\left(0,(C_\theta^{-1} x)^T \theta\right)} g(\theta,s) &= g\left(\theta, s - (C_\theta^{-1} x)^T \theta \right),
\end{align*}
it is enough to compute the reconstruction kernel $\Psi^\gamma_0 = \Psi^\gamma$ only for the point $x^* = 0$ for a mollifier with $\delta^\gamma(\cdot,x) = T^x \delta^\gamma(\cdot,0) = \delta^\gamma (\cdot - x,0)$, e.g., the Gaussian $\delta_G^\gamma$ with 
\begin{align*}
\delta_G^\gamma(y,x) = \frac{1}{2\pi \gamma^2} \exp \left( - \frac{\lVert y - x \rVert^2}{ (2 \gamma^2)} \right).
\end{align*}

This leads to the dynamic filtered backprojection formula
\begin{align}
\label{eq:dynFBP}
f^\gamma(x) 
&= \langle g, T^{\left(0,(C_\theta^{-1} x)^T \theta\right)} \Psi^\gamma \rangle_{L^2(S^1 \times \R)} \\
&= \int_{S^1} \int_{\R} g(\theta,s) \Psi^\gamma\left(\theta, s-(C_\theta^{-1} x)^T \theta\right)\, \mathrm{d}s \, \mathrm{d}\theta.
\end{align}


\vspace*{2ex}

\begin{theorem}{(Reconstruction Kernel for the Gaussian Mollifier)}
The reconstruction kernel for the Gaussian mollifier is given by
\begin{align*}
\psi^\gamma(\theta, s) =  \frac{\lvert \det C_\theta \rvert \lvert h(\theta) \rvert}{4 \pi^2 \gamma^2 \lVert C_\theta^{-*} \theta \rVert^2}  \left( 1 -  \frac{\sqrt{2}\left[s + (C_\theta^{-1} b_\theta)^T \theta\right]}{\gamma \lVert C_\theta^{-*} \theta \rVert} D\left( \frac{s + (C_\theta^{-1} b_\theta)^T \theta}{\sqrt{2} \gamma \lVert C_\theta^{-*} \theta \rVert} \right) \right)
\end{align*}
with the Dawson integral
\begin{align*}
D(s) = e^{-s^2} \int_0^s e^{t^2} \, \mathrm{d}t.
\end{align*}
\end{theorem}

\vspace*{2ex}

\begin{proof}
    The computation can be done as in the static case~\cite{LouisDiffusionRec} by using \eqref{eq:recKernel} and equations 7.4.6 and 7.4.7 from~\cite{mathematicalFunctions}. The result can also be found in~\cite{HahnDis}.
\end{proof}

\vspace*{2ex}

By applying the same kind of discretization as in the static case,  we altogether obtain the following pseudo code:

\vspace*{2ex}

\begin{lstlisting}[numbers=left, caption={pseudo code for the dynamic filtered backprojection}, label=alg1:back, mathescape=true, language=Pascal, basicstyle = \footnotesize, lineskip=2pt]
$\mathbf{Input}$: $p \text{: number of angles, } \tilde{q} \text{: uneven number of offset-values, } \Psi^\gamma \text{: dynamic filter,}$
$\text{numpix: number of pixels, }  g(\theta_j,s_k) \text{: sinogram for }\theta_i = (\cos(\varphi_i), \sin(\varphi_i))^T, $ 
$ \varphi_i = \frac{i\pi}{p}, s_k = hk \text{ for } h=1/q, k=0,\ldots,\tilde{q}-1, q = (\tilde{q}-1)/2,  i = 0,...,p-1$

1.for $i$ = $0,\ldots,p-1$ do
  	for $k$ = $0, \ldots, \tilde{q}-1$ do
  		$v_{i,k}$ = $h \sum_{j = 0}^{\tilde{q}-1} \Psi^\gamma(s_k-s_j)g(\theta_i,s_j)$
     end	
  end
  
2.for $i$ = $0, \ldots, \text{numpix}-1$ do
	$x_1$ = $1 - \frac{2}{\tilde{q}}(i + 1/2)$
 	for $j$ = $0, \ldots, \text{numpix}-1$ do
 		$x_2$ = $-1 + \frac{2}{\tilde{q}}(j + 1/2)$
 		$\text{sum}$ = 0
 		if $x_1^2 + x_2^2 > 1$ then
 			pass
 		else
 			for $l$ = $0,\ldots,p-1$ do
 				$s$ = $(C_\theta^{-1} x)^T \theta_l$
 				$k$ = $\lfloor sq \rfloor + q$
 				$\mu$ = $sq - k + q$
 				$\text{sum}$ += $(1-\mu)v_{l,k} + \mu v_{l,k+1}$
 			end
 		$f^{\mathrm{FBI}}_{i,j}$ = $\frac{2\pi}{p} \cdot \text{sum}$
 	end
  end
  $\mathbf{Output}$: $f^{\mathrm{FBI}}$
\end{lstlisting}

\vspace*{2ex}

\section{Numerical Results}
\label{s:results}

In the following, we test our method first for two simulated phantoms, the first one performing a shift and the second one a stretch. In both cases, landmark detection is performed via deep learning with the network described in Section \ref{sec:cnn6}. Then we test our method on real data, where we perform landmark detection by hand.

\vspace*{2ex}

\subsection{Simulated Data}
\paragraph*{Synthetic Measurements}
Both simulated phantoms are of size $512 \times 512$, contain one rectangle with greyscale values in $[0,1]$, and perform an affine linear motion with constant speed as in \eqref{eq:affineMotion}. In the first case, the rectangle moves with a simple shift (rigid body motion) defined in the pixel domain by
\begin{align*}
C_1(t_\theta) = I, \quad b_1(t_\theta)_{pix} = \frac{t_\theta}{N-1} 
\begin{pmatrix}
51 \\ 
51
\end{pmatrix} 
\end{align*}
In the second case, the rectangle is stretched as described by the motion

\begin{align*}
C_{2,\theta} = C_2(t_\theta) = I + \frac{t_\theta}{N-1} (C - I), \quad b_2(t_\theta) = 0
\end{align*}
for $t_\theta \in \{0,\ldots,N-1\}$, $N=450$ and 
$C = \begin{pmatrix}
2 & 0 \\ 
0 & 1
\end{pmatrix}.$

Next, we compute the dynamic sinogram as well as the two static sinograms for both the initial and the final position to calculate the inexactness $\eta(s,\varphi)$ for 300 offset values and 450 angles, and add uniformly distributed noise in $[-0.02,0.02]$. These yield the input to our hybrid three-step reconstruction scheme.

\begin{figure}[t]
\centering
\caption{Landmark detection on reconstructions from RESESOP-Kaczmarz without noise}
\includegraphics[width=\textwidth]{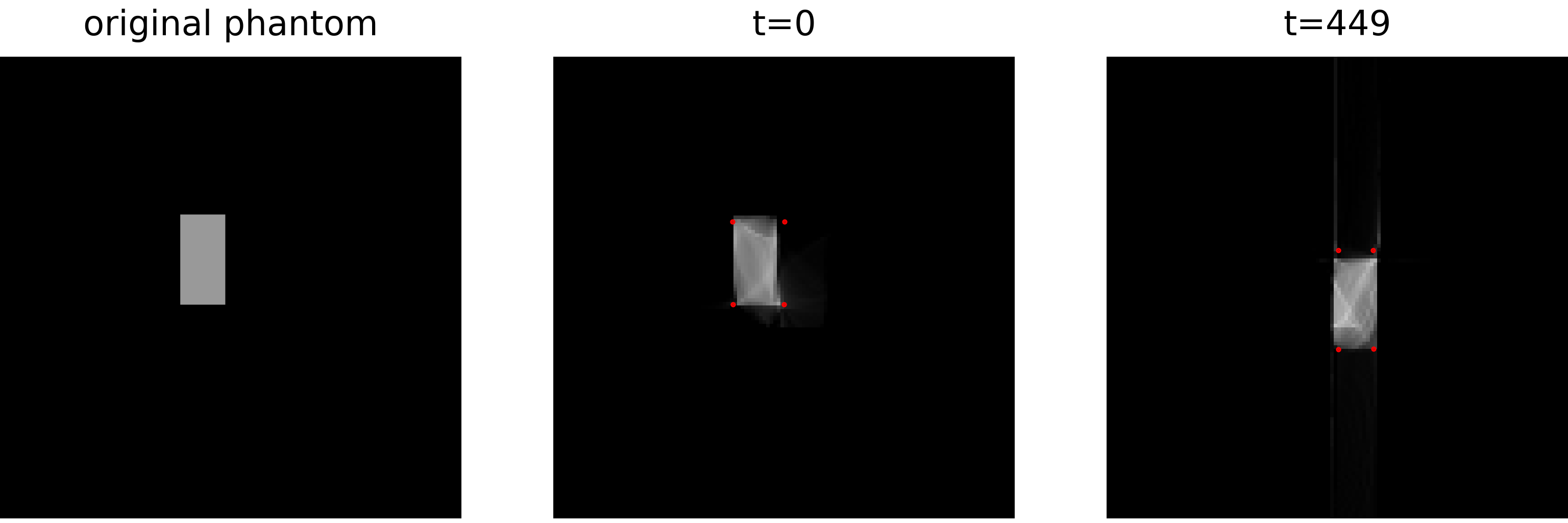}
\label{fig06:landmark}
\end{figure}

\vspace*{2ex}

\paragraph*{Reconstructions}

We reconstruct the object twice, once at the start time and again at the final time point, using three iterations of the RESESOP-Kaczmarz algorithm with the calculated inexactness. Note that we use an image size of only $128 \times 128$ pixels for these reconstructions to reduce the computation time. For all our reconstructions with RESESOP-Kaczmarz we choose $\tau = 1.00001$.
As can be seen in Figure~\ref{fig06:landmark}, the reconstructions are not satisfactory, but the quality is sufficient to see the shift of the object.
These reconstructions can now be passed on to the trained CNN6 network described in Section~\ref{sec:cnn6} to automatically detect the landmarks (here: corners), the results are shown in Figure~\ref{fig06:landmark}.

\begin{figure}[t]
\centering
\caption{Shifted rectangle without noise}
\includegraphics[width=\textwidth]{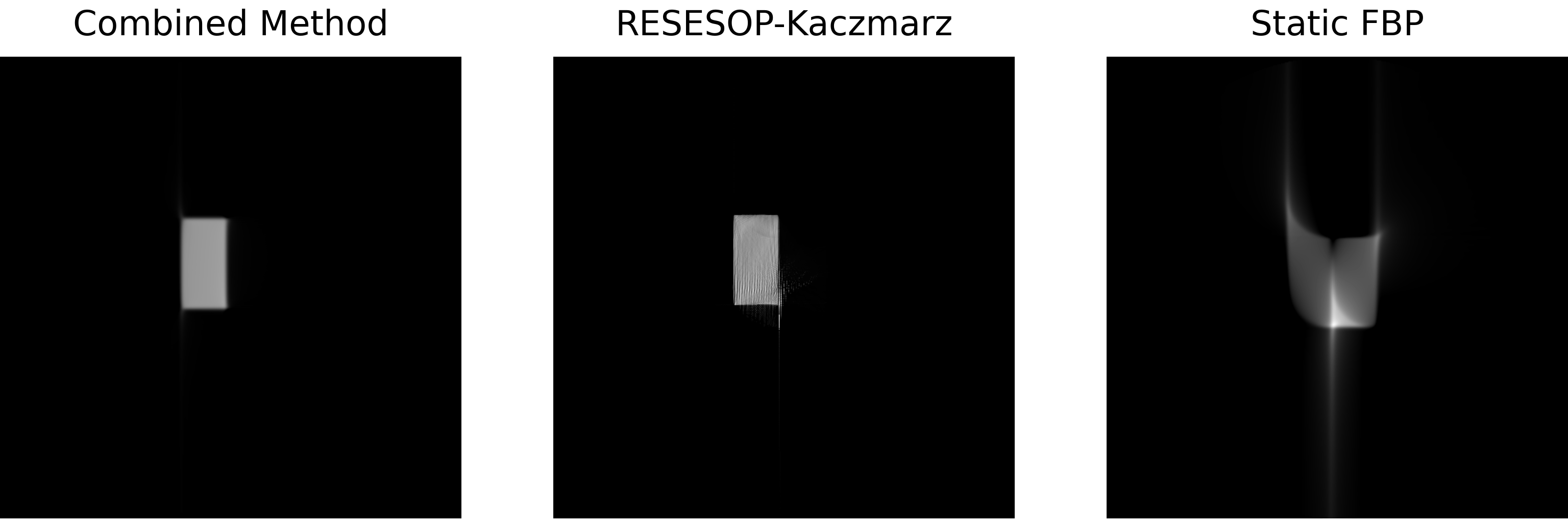}
\label{fig06:compare}
\end{figure}

\begin{figure}[h]
\centering
\caption{Shifted rectangle with noise}
\includegraphics[width=\textwidth]{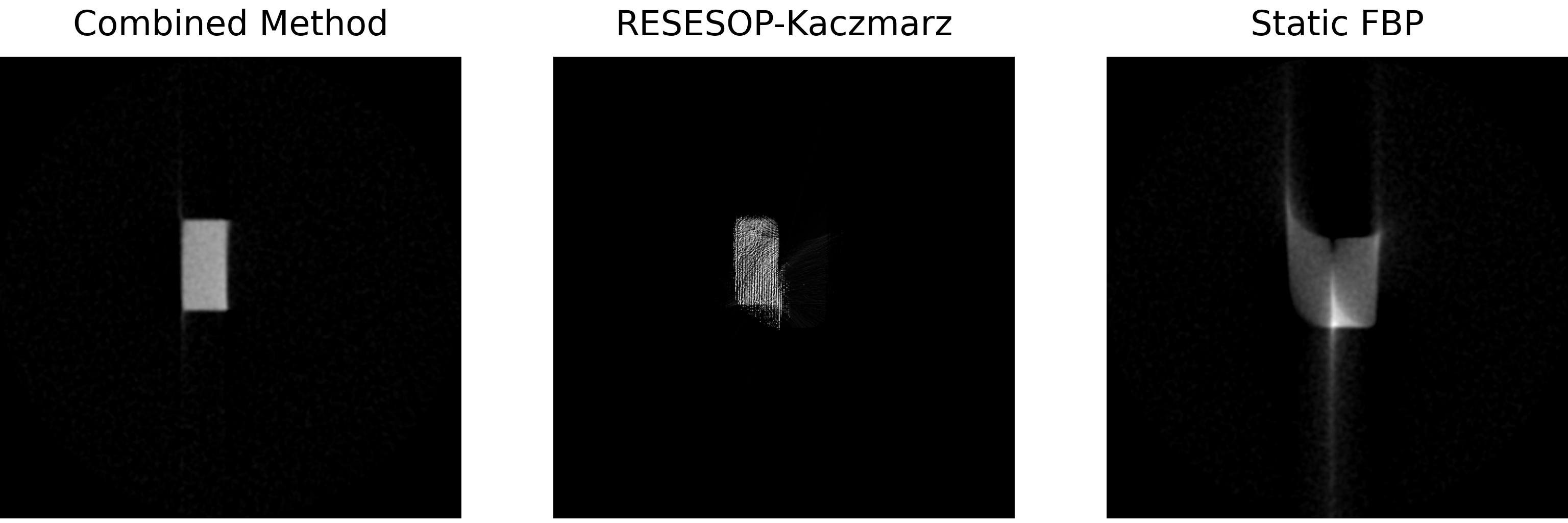}
\label{fig06:compareNoise}
\end{figure}

In the first case the shift is computed as the mean of the difference between the corner coordinates from the end point and the corner coordinates from the start point. In the second case, the entries of matrix $C$ are computed as a least-squares solution of the respective system of linear equations.

To avoid committing an inverse crime we choose the size of all reconstruction as $487 \times 487$, which differs from the original phantom size of $512 \times 512$. Some basic examples of inverse crimes and how to avoid them can be found in \cite[Chapter II]{inverseCrime}.

For comparison, we computed RESESOP-Kaczmarz reconstructions with $30$ iterations.
While the reconstruction of the RESESOP-Kaczmarz algorithm still has a small motion artifact under the lower right corner as well as inside the rectangle, the reconstruction of the combined method turns out to be a slightly blurred version of the ground truth 
(see Figures~\ref{fig06:compare} and~\ref{fig06:compareNoise}). Our hybrid method is thus able to recover the object without using the actual motion as an input.

\vspace{1ex}

In the case of the stretched rectangle (figure~\ref{fig06:compareNoiseMatrix}), the output of the combined method has stronger artifacts since the landmarks were not detected precisely, and therefore also the motion estimation was faulty.

\vspace{1ex}

Considering the run time, our combined method leads to a speed-up and only needs a little more than half of the computational time that is needed for 30 iterations of RESESOP-Kaczmarz.

\begin{figure}[t]
\centering
\caption{Stretched Rectangle with Noise}
\includegraphics[width=0.85\textwidth]{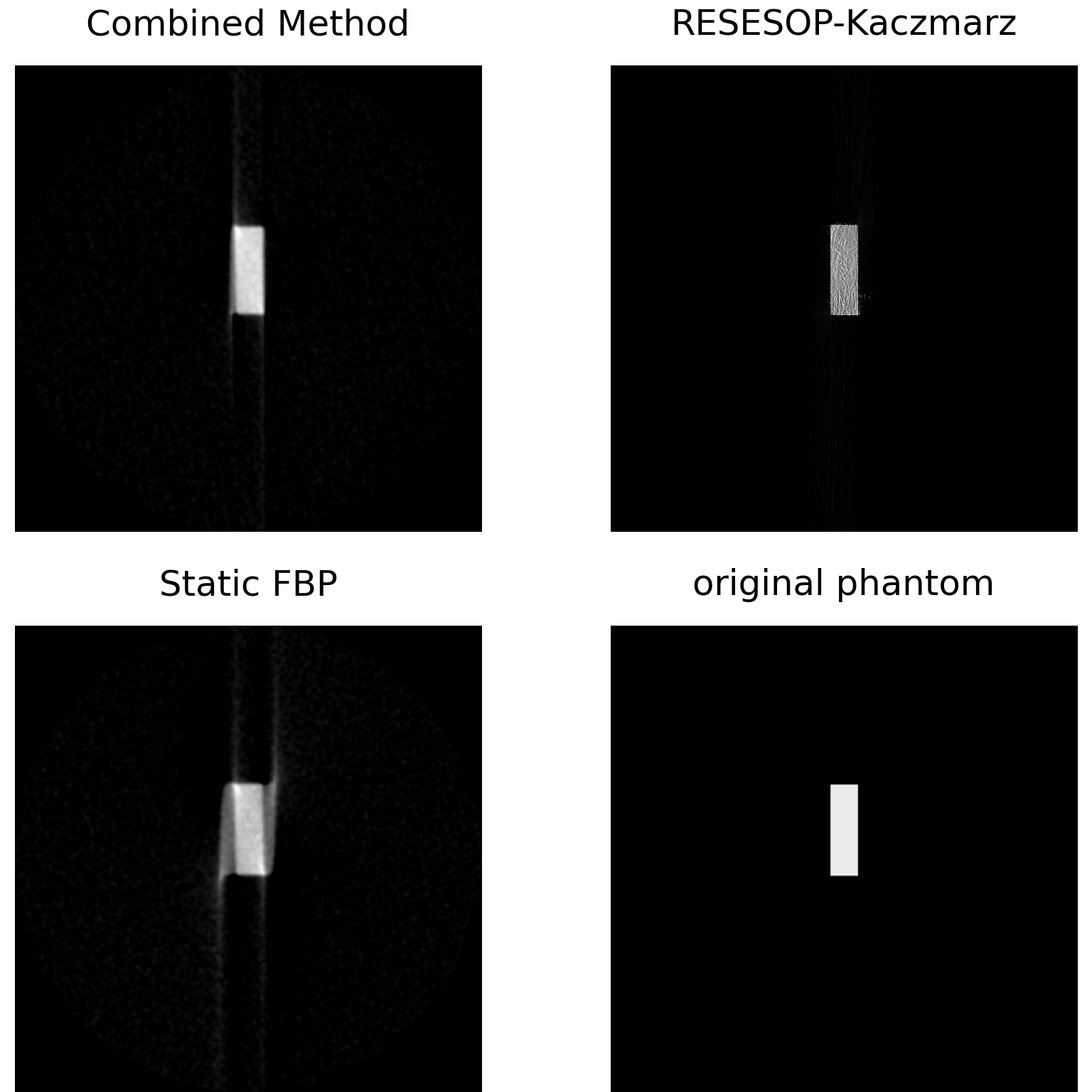}
\label{fig06:compareNoiseMatrix}
\end{figure}

\vspace*{2ex}

\subsection{Real Data}
\label{sec06:realData}

The data was measured by a cone-beam CT scanner designed and constructed in-house in the Industrial Mathematics Computed Tomography Laboratory at the University of Helsinki and was preprocessed using the Helsinki Tomography Toolbox for Matlab~\cite{heltomo}. The measurement settings were almost the same as in the open access CT dataset \cite{seashell} of a scanned seashell, where the following measurement description is based on.

The scanner consists of a motorized rotation stage (Thorlabs CR1-Z7), a molybdenum target X-ray tube (Oxford Instruments XTF5011), and a 12-bit, 2240x2368 pixel, energy-integrating flat panel detector (Hamatsu Photonics C7942CA-22).

360 X-ray projections were acquired using an angle increment of 1 degree. The X-ray source voltage and tube current were set at 45 kV and 1 mA, respectively. The exposure time of the flat panel detector was set to 1200 ms.

Two correction images were acquired before scanning the sample. A dark current image was created by averaging 255 images taken with the X-ray source off. A flat-field image was created by averaging 255 images taken with the X-ray source switched on with no sample placed in the scanner. After the scan, dark current and flat-field corrections \cite{seibert} were applied to each projection image using the Hamamatsu HiPic imaging software version 9.3.

Fan-beam sinogram data was created by using only the measured intensity values from the central plane of the X-ray cone beam.
Due to a slightly misaligned center of rotation in the scanner, the CT reconstructions can appear blurry. It was empirically observed that this problem can be compensated for quite well by shifting each projection to the left by five pixels, using circular boundary conditions, before performing any other operations on the projections.

\begin{figure}[t]
\centering
\caption{\label{fig06:LandmarkLego}RESESOP-Kaczmarz Reconstructions for the first and last time step after three iterations}
\includegraphics[width=\textwidth]{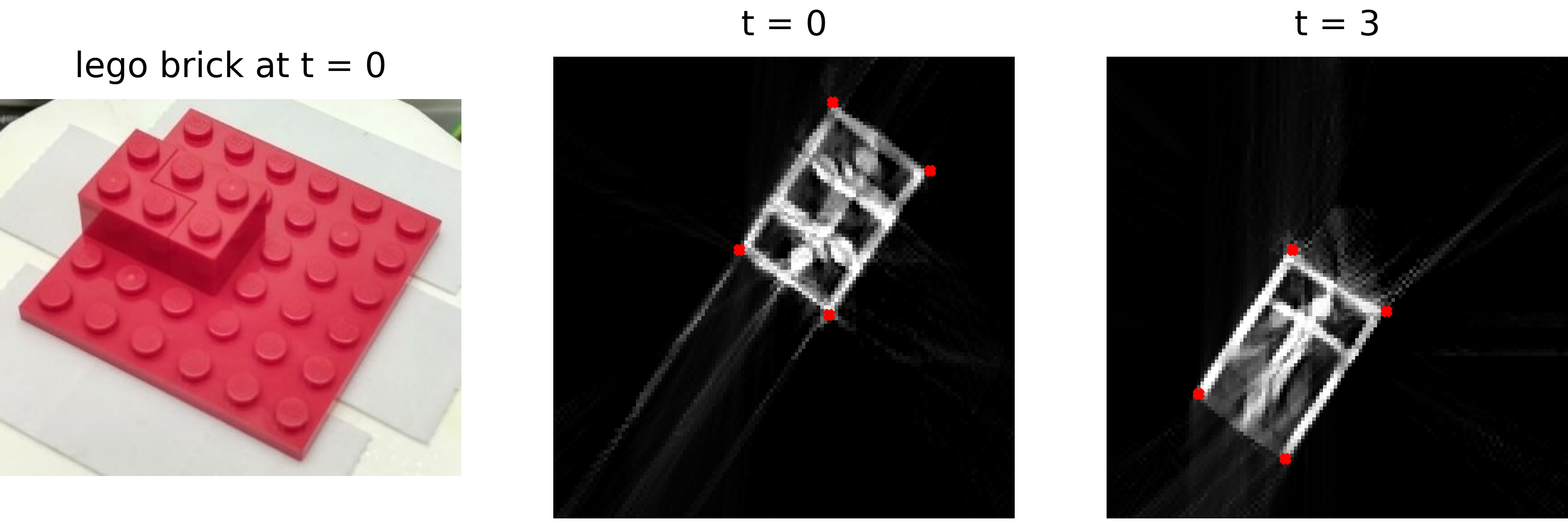}
\end{figure}
We scanned a composite Lego brick, shown in figure~\ref{fig06:LandmarkLego}, with these settings for four different time steps. For each measurement, we moved the brick one knob forward such that it reached the other end at the last time step. Then we computed the parallel-beam sinograms ranging from 0 to 180 degrees with the inbuild Matlab binning algorithm 'fan2para' and built the dynamic sinogram by taking the first 45 degrees from the first time step, 46 to 90 degrees from the second time step, and so on.

\begin{figure}[t]
\centering
\caption{\label{fig06:ComparisonLego}Comparison of the 3 Reconstruction Methods in relation to the FBP computed with the static sinogram}
\includegraphics[width=0.85\textwidth]{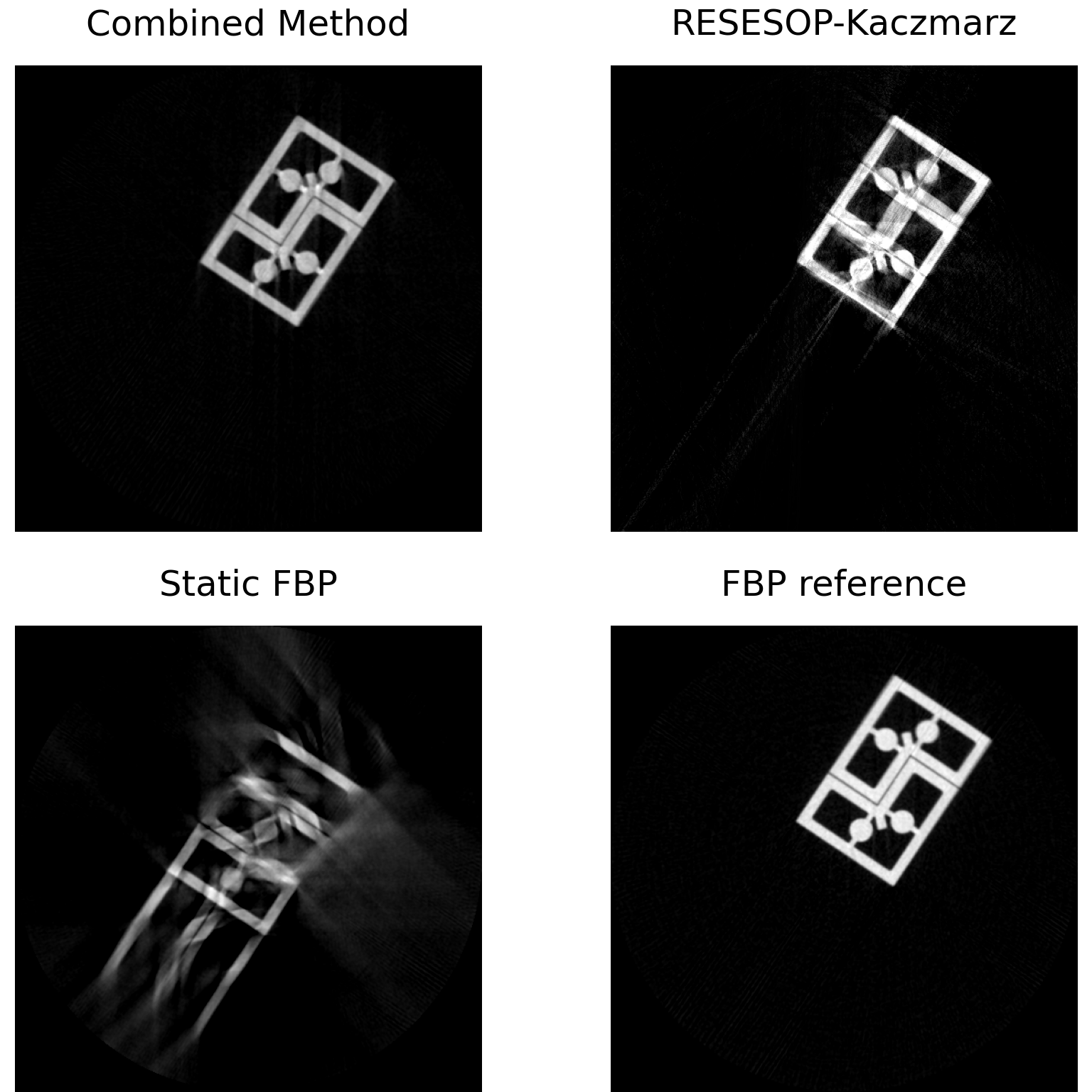}
\end{figure}

As inexactness for the RESESOP-Kaczmarz algorithm, we took the absolute differences between the dynamic sinogram and the sinogram of the first resp. last position of the lego brick.
We estimated the noise by taking the mean of the left blank part of the dynamic sinogram.

Three iterations of RESESOP-Kaczmarz to determine the initial and final state, and landmark detection by hand lead to the reconstructions of size $128 \times 128$, shown in figure~\ref{fig06:LandmarkLego}.

Afterwards, the mean of each coordinate shift yield the total movement of the Lego brick. To compute the reconstruction of the combined algorithm with dynamic filtered backprojection, we assumed that the movement took place in four steps with the same amount of projections per step. The results of size $512 \times 512$ are displayed in Figure~\ref{fig06:ComparisonLego}.
As in the experiments with the simulated data, the reconstruction by the combined method shows fewer artifacts and is about twice as fast as a pure RESESOP-Kaczmarz reconstruction.

\vspace*{2ex}

\section{Conclusion}

In this article, we proposed a proof of concept for a combination of RESESOP-Kaczmarz, landmark detection and filtered backprojection, i.e., a hybrid reconstruction technique that incorporates classical as well as data-driven methods for computerized tomography, to image an object undergoing an affine linear motion with constant speed during data acquisition. The motion is not known a priori, but estimated during the reconstruction. 
In addition, our combined method leads to a faster and more accurate solution of dynamic computerized tomography provided the respective model inexactness, stemming from working with the static model, is accessible and if the landmark detection works well.

By consequence, future research is dedicated to estimating the inexactness directly from the data to provide a realistic initial setting for our pipeline. Also, the landmark detection can be generalized towards more complex objects

In the future, it will be necessary to find ways to estimate the modeling inexactness given only the dynamic sinogram. Furthermore, it would be interesting to train neural networks performing landmark detection for a broader class of objects.

Our method could also be extended to more complex motions by interpolating the motion piecewise by affine linear motions with constant speed. To achieve this, it would be possible to estimate the inexactness at multiple different time points.

\vspace*{2ex}

\textbf{Acknowledgements:} 
G.S.~and A.W.~acknowledge support by the German
Research Foundation (DFG; Project-ID 432680300 ?
SFB 1456), and thank Bernadette Hahn-Rigaud and Alice Oberacker for fruitful discussions and help with the implementation.

\bibliographystyle{siam}
\bibliography{lit}

\end{document}